\documentclass{amsart}
\usepackage{amssymb, amsthm, amsmath, amsfonts,amscd}
\usepackage{graphics}
\usepackage{hyperref}
\usepackage[all]{xy}
\usepackage{enumerate}
\usepackage[mathscr]{eucal}
\usepackage{bbm}

\providecommand{\U}[1]{\protect\rule{.1in}{.1in}}
\def\theenumi{\arabic{enumi}}

\def\theenumii{\alph{enumii}}
\def\p@enumii{\theenumi.}

\def\theenumiii{\arabic{enumiii}}
\def\p@enumiii{(\theenumi)(\theenumii)}

\def\p@enumiv{\p@enumiii.\theenumiii}
\theoremstyle{plain}
\newtheorem{theorem}{Theorem}[section]
\newtheorem{lemma}[theorem]{Lemma}
\newtheorem{proposition}[theorem]{Proposition}

\newtheorem{corollary}[theorem]{Corollary}
\newtheorem{conjecture}[theorem]{Conjecture}
\numberwithin{equation}{section}

\theoremstyle{definition}

\newtheorem{definition}[theorem]{Definition}

\newtheorem{remark}[theorem]{Remark}

\newtheorem{thmab}{Theorem}

\renewenvironment{proof}[1][\proofname]{{\bfseries #1\\}}{\qed}

\DeclareMathOperator{\Mod}{-Mod}
\DeclareMathOperator{\module}{-mod}
\DeclareMathOperator{\Hom}{Hom}
\DeclareMathOperator{\End}{End}

\DeclareMathOperator{\hd}{hd}

\DeclareMathOperator{\gd}{gd}

\DeclareMathOperator{\td}{td}
\DeclareMathOperator{\reg}{reg}
\DeclareMathOperator{\coker}{coker}
\DeclareMathOperator{\im}{im}
\DeclareMathOperator{\Ext}{Ext}
\DeclareMathOperator{\Res}{Res}

\DeclareMathOperator{\depth}{depth}

\newcommand{\as}{\text{*}}
\newcommand{\dt}{\bullet}

\newcommand{\C}{{\mathscr{C}}}

\newcommand{\Z}{{\mathbb{Z}}}

\newcommand{\N}{\mathbb{N}}
\newcommand{\FI}{{\mathscr{FI}}}
\newcommand{\mi}{\mathfrak{m}}

\newcommand{\fiHom}{\mathscr{H}om}
\newcommand{\fiExt}{{\mathscr{E}xt}}
\newcommand{\arXiv}[1]{\href{http://arxiv.org/abs/#1}{\nolinkurl{arXiv:#1}}}
\newcommand{\arXivV}[2]{\href{http://arxiv.org/abs/#1}{\nolinkurl{arXiv:#1v#2}}}

\title{On the Degree-Wise Coherence of $\FI_G$-modules}

\author{Eric Ramos}
\address{Department of Mathematics, University of Wisconsin - Madison.}
\email{eramos@math.wisc.edu}

\thanks{The author was supported by NSF-RTG grant 1502553.}

\begin{document}

\begin{abstract}
In this work we study a kind of coherence condition on $\FI_G$-modules, which generalizes the usual notion of finite generation. We prove that a module is coherent, in the appropriate sense, if and only if its generators, as well as its torsion, appears in only finitely many degrees. Using this technical result, we prove that the category of coherent $\FI_G$-modules is abelian, independent of any assumptions on the group $G$, or the coefficient ring $k$. Following this, we consider applications towards the local cohomology theory of $\FI_G$-modules, introduced by Li and the author in \cite{LR}.\\
\end{abstract}

\maketitle

\section{Introduction}

Let $\FI$ be the category whose objects are the sets $[n] := \{1,\ldots, n\}$, and whose morphisms are injections. An \textbf{$\FI$-module} over a commutative ring $k$ is a functor from the category $\FI$ to the category of $k$-modules. $\FI$-modules were first introduced by Church, Ellenberg, and Farb as a way to study stability phenomena common throughout mathematics \cite{CEF}. Following this work, representations of various other categories were studied by a large collection of authors. See \cite{W}, \cite{SS}, \cite{SS2}, \cite{PS}, for examples of this work. In this paper, we will be concerned with modules over a category which naturally generalizes $\FI$, $\FI_G$.\\

Let $G$ be a group. Then the category $\FI_G$ is that whose objects are the sets $[n]$, and whose morphisms $(f,g):[n] \rightarrow [m]$ are pairs of an injection $f$ with a map of sets $g:[n] \rightarrow G$. If $(f,g)$ and $(f',g')$ are two composable morphisms in $\FI_G$, then we define
\[
(f,g)\circ(f',g') := (f\circ f', h), \hspace{1cm} h(x) = g'(x)\cdot g(f'(x))
\]
If $G = 1$ is the trivial group, then it is easily seen that $\FI_G$ is equivalent to the category $\FI$. If, instead, we assume that $G = \Z/2\Z$, then $\FI_G$ is equivalent to the category $\FI_{BC}$ first introduced by Wilson in \cite{W}. An \textbf{$\FI_G$-module} over a commutative ring $k$ is defined in the same way as it was for $\FI$-modules. $\FI_G$-modules were first introduced by Sam and Snowden in \cite{SS2}.\\

For much of this paper, we will be concerned with the category $\FI_G\Mod$ of $\FI_G$-modules. It is immediate that $\FI_G\Mod$ is an abelian category with the usual abelian operations being computed on points. Because of its close connections with the category $k\Mod$, one may define many properties of $\FI_G$-modules which are analogous to properties of $k$-modules. One such property, which is most important to us, is finite generation. We say that an $\FI_G$-module $V$ is \textbf{finitely generated} if there exists a finite set $\{v_i\} \subseteq \sqcup_{n\geq 0} V([n])$, which no proper submodule contains. Perhaps the most significant fact about finitely generated $\FI_G$-modules is that they are often times Noetherian.\\

\begin{theorem}[Corollary 1.2.2 \cite{SS2}]
Let $G$ be a polycyclic-by-finite group, and let $k$ be a Noetherian ring. Then submodules of finitely generated $\FI_G$-modules are themselves finitely generated.\\
\end{theorem}

Note that another way of thinking of the above theorem is that the category $\FI_G\module$ of finitely generated $\FI_G$-modules is abelian under sufficient restrictions on $k$ and $G$. The hypotheses of the above theorem are currently the most general known. It is conjectured that $G$ being polycyclic-by-finite is also necessary for the Noetherian property to hold \cite{SS2}. One of the main goals of this paper is to argue that many theoretical constructions in the theory of $\FI_G$-modules can actually be done independent of the Noetherian property. Instead, we argue that \textbf{degree-wise coherence} is often sufficient.\\

We say that an $\FI_G$-module is degree-wise coherent if there is a set (not necessarily finite) $\{v_i\} \subseteq \sqcup_{n\geq 0} V([n])$ such that:
\begin{enumerate}
\item no proper submodule contains $\{v_i\}$, and there is some $N \gg 0$ such that $\{v_i\} \subseteq \sqcup_{n = 0}^{N}V([n])$. In this case we say that $V$ is \textbf{generated in finite degree};
\item the module of relations between the elements $\{v_i\}$ is itself generated in finite degree (see Definition \ref{pres}).\\
\end{enumerate}

One can think about the above definition in the following way. Instead of requiring that our module have finitely many generators, we only require that it admits a generating set whose elements appear in at most finitely many degrees. In addition, we also require that these generators have relations which are bounded in a similar sense. The significance of this condition traces its origins to the paper \cite{CE}, although they do not use the same terminology. Following this work, degree-wise coherent modules were studied more deeply by the author in \cite{R}. The first goal of this paper will be to understand the connection between being degree-wise coherent, and having finite torsion.\\

We say an element $v \in V([n])$ is \textbf{torsion} if there is some morphisms $(f,g):[n] \rightarrow [m]$ in $\FI_G$, such that $V(f,g)(v) = 0$. The \textbf{torsion degree} of an $\FI_G$-module is the quantity,
\[
\td(V) := \sup\{n \mid V_n \text{ contains a torsion element.}\}
\]
It was first observed by Church and Ellenberg that degree-wise coherent $\FI$-modules will necessarily have finite torsion degree \cite[Theorem D]{CE}. It was then later shown by the author that the same statement was true for $\FI_G$-modules \cite[Theorem 3.19]{R}. More recently, Li has conjectured that the converse of this statement was true as well \cite{L3}. In this paper, we will prove this conjecture in the affirmative.\\

\begin{thmab}
Let $G$ be a group, and let $k$ be a commutative ring. If $V$ is an $\FI_G$-module which is generated in finite degree, then $V$ is degree-wise coherent if and only if $\td(V) < \infty$.\\
\end{thmab}

As a first application of the above technical theorem, we will be able to show that degree-wise coherent modules form an abelian category.\\

\begin{thmab}\label{abel}
Let $G$ be a group, and $k$ a commutative ring. Then the category $\FI_G\Mod^{coh}$ of degree-wise coherent modules is abelian.\\
\end{thmab}

This theorem was recently proven independently by Li in his note \cite[Proposition 3.4]{L3}. One immediately sees that the above theorem is independent of the ring $k$, as well as the group $G$. As stated previously, working in the category $\FI_G\Mod^{coh}$ often has benefits which the category $\FI_G\module$ does not permit. Perhaps the most explicit of these benefits is the existence of infinite shifts, which we discuss below. Of course, one should note that there are also benefits which are exclusive to finitely generated modules. The most obvious of these is the ability to do explicit computations.\\

Much of the remainder of the paper is dedicated to showing how well known theorems about finitely generated $\FI_G$-modules will continue to hold in the category $\FI_G\Mod^{coh}$. In particular, we focus on generalizing the local cohomology theory of $\FI_G$-modules, introduced by Li and the author in \cite{LR}.\\

If $V$ is an $\FI_G$-module, then the \textbf{0-th local cohomology functor} is defined by
\[
H^0_\mi(V) := \text{ the maximal torsion submodule of $V$.}
\]
$H^0_\mi$ is a left exact functor, and we denote its derived functors by $H^i_\mi$. Section \ref{locsection} is largely dedicated to arguing that the theorems of \cite{LR} will continue to hold in $\FI_G\Mod^{coh}$. One of the main results of \cite{LR}, is that whenever $V$ is finitely generated there is a complex $\C^\dt V$ which computes $H^i_\mi$ (see Definition \ref{nonfunctcomplex}). One problem with this complex, is that it's not functoral in $V$. Allowing ourselves to work in the category $\FI_G\Mod^{coh}$, we can fix this issue using the infinite shift.\\

Let $\iota:\FI_G \rightarrow \FI_G$ be the functor defined by the assignments,
\[
\iota([n]) = [n+1], \hspace{1cm} \iota((f,g):[n] \rightarrow [m]) = (f_+,g_+)
\]
where
\[
f_+(x) = \begin{cases} f(x) &\text{ if $x < n+1$}\\ m+1 &\text{ otherwise}\end{cases}, \hspace{1cm} g_+(x) = \begin{cases} g(x) &\text{ if $x < n+1$}\\ 1 &\text{ otherwise.}\end{cases}
\]
The shift functor $\Sigma$ is defined to be
\[
\Sigma(V) := V \circ \iota.
\]
We write $\Sigma_b$ to denote the $b$-th iterate of $\Sigma$. In Section \ref{shift}, it is show that there is a commutative diagram for all $b \geq 1$,
\[
\begin{CD}
V @>>> \Sigma_{b+1}\\
@| @AAA\\
V @>>> \Sigma_{b}
\end{CD}
\]
The \textbf{infinite shift} $\Sigma_\infty$ is the directed limit of the right column of this diagram. That is,
\[
\Sigma_\infty V := \lim_{\rightarrow} \Sigma_b V.
\]
The collection of maps $V \rightarrow \Sigma_b$ in the above diagram induce a morphism $V \rightarrow \Sigma_\infty V$. The \textbf{infinite derivative} is defined to be the cokernel of this map
\[
D_\infty V := \coker(V \rightarrow \Sigma_\infty V).
\]
One should observe that is rarely ever the case that the infinite derivative or the infinite shift are finitely generated. We will see, however, that if $V$ is degree-wise coherent, then the same is true of both $\Sigma_\infty V$ and $D_\infty V$. It is shown in Section \ref{infderv} that the infinite derivative functor is right exact. We use $H_i^{D_\infty^b}$ to denote the $i$-th left derived functor of the $b$-th iterate of $D_\infty$. The main result of the final section of the paper is the following.\\

\begin{thmab}
Let $V$ be a degree-wise coherent $\FI_G$-module of dimension $d < \infty$ (see Definition \ref{dim}). Then there are isomorphisms for all $i \geq 1$,
\[
H_i^{D_\infty^{d+1}}(V) \cong H^{d+1-i}_\mi(V).
\]
\text{}\\
\end{thmab}

One can think of the above theorem as a kind of local duality for $\FI_G$-modules, in so far as it describes the equivalence of local cohomology with the derived functors of some right exact functor. We have already discussed the fact that the functor $D_\infty$ does not exist within the category of finitely generated modules, and therefore the above represents a means of uniformly describing local cohomology modules in a way which is inaccessible by simply working with finitely generated modules.\\

\section*{Acknowledgments}
The author would like to send thanks to Liping Li, Rohit Nagpal, and Andrew Snowden for the many conversations which inspired this work. The author would also like the acknowlege the generous support of the National Science Foundation through NSF-RTG grant 1502553.\\
\section{Preliminaries}

\subsection{Elementary Definitions}

Let $G$ be a group, and let $k$ be a commutative ring.\\

\begin{definition}
The category $\FI_G$ is that whose objects are the finite sets $[n] := \{1,\ldots,n\}$, and whose morphisms are pairs $(f,g):[n] \rightarrow [m]$, where $f:[n] \rightarrow [m]$ is an injection of sets and $g:[n] \rightarrow G$ is a map of sets. For two composable morphisms $(f,g),(h,g')$, we define
\[
(f,g) \circ (h,g') := (f \circ h,g'')
\]
where $g''(x) = g'(x)\cdot g(h(x))$. For each non-negative integer $n$, we denote the group of endomorphisms $\End_{\FI_G}([n]) = \mathfrak{S}_n \wr G$ by $G_n$.\\

An \textbf{$\FI_G$-module} over $k$ is a covariant functor $V:\FI_G \rightarrow k\Mod$. We use $V_n$ to denote the $k$-module $V([n])$. For any $\FI_G$-morphism $(f,g):[n] \rightarrow [m]$ we write $(f,g)_\as$ for the map $V(f,g)$. We call these maps the \textbf{induced maps} of $V$, and in the case where $n < m$ we say that $(f,g)_\as$ is a \textbf{transition map} of $V$.\\

Given any $\FI_G$-module $V$, its \textbf{degree} is the quantity,
\[
\deg(V) := \sup\{n \mid V_n \neq 0\} \in \N \cup \{\pm \infty\}
\]
where we use the convention that the supremum of the empty set is $-\infty$.\\
\end{definition}

We note that the category of $\FI_G$-modules and natural transformations $\FI\Mod$ is abelian. Indeed, one computes kernels and cokernels in a pointwise fashion. One nice feature of $\FI_G$-modules is that many properties of $k$-modules have natural analogs. Perhaps the most significant of these properties is finite generation.\\

\begin{definition}
Let $V$ be an $\FI_G$-module. We say that $V$ is \textbf{finitely generated} if there is a finite collection $S \subseteq \sqcup_{n \geq 0} V_n$ which no proper submodule of $V$ contains. We denote the category of finitely generated $\FI_G$-modules by $\FI_G\module$.\\
\end{definition}

Finitely generated $\FI_G$-modules were first studied by Sam and Snowden in \cite{SS2}. Prior to this, the case wherein $G = 1$ was studied by Church, Ellenberg, Farb, and Nagpal in \cite{CEF}, and \cite{CEFN}. This case was also featured prominently in the work of Sam and Snowden \cite{SS3}. We note that Church, Ellenberg, Farb, and Nagpal refer to these modules as being $\FI$-modules. The case wherein $G = \Z/2\Z$ was studied by Wilson in \cite{W}. Wilson refers to these modules as being $\FI_{BC}$-modules.\\

\begin{theorem}[Corollary 1.2.2 \cite{SS2}]
Assume that $G$ is a polycyclic-by-finite group, and that $k$ is a Noetherian ring. Then the category $\FI_G\module$ is abelian. That is, submodules of finitely generated modules are finitely generated.\\
\end{theorem}

One should observe the two hypotheses of the above theorem. In this paper we will not be studying finitely generated $\FI_G$-modules, instead focusing on degree-wise coherent modules (see Definition \ref{pres}). Working with these more general modules will allow us to prove many theorems without needing to restrict the ring $k$ or the group $G$. One goal of this paper is to argue that degree-wise coherence is a more natural condition than finite generation in many contexts.\\

\begin{definition}\label{pres}
Let $r \geq 0$ be an integer. The \textbf{principal projective $\FI_G$-module generated in degree $r$} $M(r)$ is defined on points by
\[
M(r)_n := k[\Hom_{\FI_G}([r],[n])],
\]
where $k[\Hom_{\FI_G}([r],[n])]$ is the free $k$-module with basis labeled by the set $\Hom_{\FI_G}([r],[n])$. The induced maps of this module act by composition on the basis vectors. More generally, if $W$ is a $kG_r$-module, then we define the \textbf{free $\FI_G$-module relative to $W$} $M(W)$ by the assignments
\[
M(W)_n := k[\Hom_{\FI_G}([r],[n])] \otimes_{kG_r} W.
\]
The induced maps of this module act by composition in the first component. In this case, we say that $M(W)$ is generated in degree $r$. Direct sums of modules of either of these two types will generally be referred to as \textbf{free modules}. The \textbf{generating degree} of a free module is the supremum of the generating degrees of its free summands.\\

We say that a module $V$ is \textbf{$\sharp$-filtered} if it admits a finite filtration
\[
0 = V^{(-1)} \subseteq \ldots \subseteq V^{(n)} = V.
\]
such that $V^{(i)}/V^{(i-1)}$ is a free module for each $i$. In this case, the integer $n$ is called the generating degree of $V$.\\

A \textbf{presentation} for a module $V$ is an exact sequence of the form,
\[
0 \rightarrow K \rightarrow F \rightarrow V \rightarrow 0,
\]
where $F$ is a free-module. If $F$ is $\sharp$-filtered with generating degree $n$, then we say that $V$ is \textbf{generated in degree $\leq n$}. If, in addition, $K$ is generated in finite degree, then we say that $V$ is \textbf{degree-wise coherent}. We denote the category of modules which are generated in finite degree by $\FI_G\Mod^{coh}$.\\
\end{definition}

Note that free modules are not always projective, although projective modules are always free. Indeed, it can be shown that for a $kG_r$-module $W$, $M(W)$ is projective as an $\FI_G$-module if and only if $W$ is projective as a $kG_r$-module. Proofs of these facts can be found in \cite{R}.

\subsection{The homology functors and regularity}

\begin{definition}
Let $V$ be an $\FI_G$-module. Then the \textbf{0-th homology functor} is defined on points by
\[
H_0(V)_n := V_n / V_{<n},
\]
where $V_{<n}$ is the submodule of $V_n$ spanned by the images of all transition maps into $V_n$. We write $H_i$ to denote the $i$-th derived functor of $H_0$.\\

The \textbf{$i$}-th homological degree of a module $V$ is the quantity
\[
\hd_i(V) := \deg(H_i(V)) \in \N \cup \{\pm \infty\}.
\]
the 0-th homological degree $\hd_0(V)$ will be referred to as the \textbf{generating degree} of the module, and is denoted by $\gd(V)$. The \textbf{regularity} of a module $V$ is 
\[
\reg(V) := \inf\{N \mid \hd_i(V) - i \leq N \forall i \geq 1\} \in \N \cup \{\pm \infty\}.
\]
\text{}\\
\end{definition}

\begin{remark}
Note that in the above definition, regularity is computed using strictly positive homological degrees. This is slightly different from how regularity is defined in classical commutative algebra. When we discuss local cohomology later in this paper, it will be explained why the above definition was chosen.\\

It is an easy check to show that the definition of $\gd(V)$ given above agrees with the notion of generating degree given in Definition \ref{pres}. It is also important that one notes the connection between the module of relations of $V$, and the first homological degree $hd_1(V)$. Given a presentation,
\[
0 \rightarrow K \rightarrow F \rightarrow V \rightarrow 0
\]
we may apply the homology functor to find,
\[
hd_1(V) \leq gd(K) \leq \max\{gd(V),hd_1(V)\}.
\]
In particular, $V$ is degree-wise coherent if and only if both $gd(V)$ and $hd_1(V)$ are finite.\\

If $V$ is acyclic with respect to the homology functors, then we define its regularity to be $-\infty$.\\
\end{remark}

The regularity of $\FI$-modules was first studied by Sam and Snowden in \cite[Corollary 6.3.5]{SS3}, in the case where $k$ is a field of characteristic 0. Following this, Church and Ellenberg provided explicit bounds on the regularity of $\FI$-modules over any commutative ring $k$ \cite[Theorem A]{CE}. The author then adapted the techniques of Church and Ellenberg to work for general $\FI_G$-modules \cite[Theorem D]{R}.\\

\begin{theorem}[\cite{CE},\cite{R}]\label{finreg}
Let $V$ be an $\FI_G$-module. Then,
\[
\reg(V) \leq \hd_1(V) + \min\{\hd_1(V),\gd(V)\} - 1.
\]
In particular, if $V$ is degree-wise coherent, then $V$ has finite regularity.\\
\end{theorem}

One notable takeaway from the work of Church and Ellenberg is that their bound is only dependent on the generating degree and first homological degree of the module. In particular, their work entirely takes place in the category $\FI\Mod^{coh}$. This philosophy was also heavily featured in \cite{R}. One goal of the present work is to develop an understanding of the category $\FI_G\Mod^{coh}$.\\

Following this work, regularity was studied Liang Gan, Li, and the author in \cite{G}, \cite{L}, \cite{L2}, and \cite{LR}. The paper \cite{LR} studied the connection between regularity and a local cohomology theory for $\FI_G$-modules, in the case where $G$ is a finite group. We will later rediscover this connection in the more general context of the current work.\\

To conclude this section, we state the theorem which classifies the homology acyclic modules.\\

\begin{theorem}[Theorem 1.3 \cite{LY}, Theorem A \cite{R}]\label{homacyclic}
Let $V$ be a degree-wise coherent module. Then the following are equivalent:
\begin{enumerate}
\item $V$ is acylic with respect to the homology functors;
\item $H_1(V) = 0$;
\item $H_i(V) = 0$ for some $i \geq 1$;
\item $V$ is $\sharp$-filtered.\\
\end{enumerate}
\end{theorem}

\subsection{The shift and derivative functors}\label{shift}

\begin{definition}
Let $\iota:\FI_G \rightarrow \FI_G$ be the functor which is defined on objects by $\iota([n]) = [n+1]$, while for each morphism $(f,g):[n] \rightarrow [m]$ we set $\iota(f,g) = (f_{+},g_{+}))$ where,
\[
f_{+}(x) := \begin{cases} f(x) &\text{ if $x \leq n$}\\ m+1 &\text{ otherwise}\end{cases}, \hspace{1cm} g_{+}(x) := \begin{cases} g(x) &\text{ if $x \leq n$}\\ 1 &\text{ otherwise.}\end{cases}
\]
The \textbf{shift functor} is defined as the composition
\[
\Sigma V := V \circ \iota.
\]
We write $\Sigma_a$ for the $a$-th iterate of $V$.\\

For each positive integer $a$, there is a natural map of $\FI_G$-modules $\tau_a:V \rightarrow \Sigma_a V$ defined on each point by the transition map $(f^n_a,\mathbf{1})_\as$, where $f^n_a:[n] \rightarrow [n+a]$ is the natural inclusion while $\mathbf{1}$ is the trivial map into $G$. The \textbf{length $a$ derivative functor} is the cokernel of this map
\[
D_aV := \coker(\tau_a)
\]
We write $D^b_a$ for the $b$-th iterate of $D_a$. In the case where $a = 1$, we will write $D := D_1$.\\
\end{definition}

The derivative functors were introduced by Church and Ellenberg in \cite{CE}, and have since seen use in \cite{R} and \cite{LY}. Later, we will consider the direct limit of all derivative functors, which we call the infinite derivative (see Definition \ref{infderv2}). We record some useful properties of the derivative and shift functors below. Proofs of these facts can be found in \cite[Proposition 3.3]{R} and \cite[Proposition 3.5]{CE}.\\

\begin{proposition}\label{dervprop}
Fix an integer $a \geq 1$. The length $a$ derivative functor and the shift functor enjoy the following properties:
\begin{enumerate}
\item If $V$ is an $\FI_G$-module which is degree-wise coherent, then the same is true of $D_aV$ and $\Sigma V$;
\item If $\gd(V) \leq d$, then $\gd(\Sigma V) \leq d$ and $\gd(D_aV) < d$;
\item $D_a$ is right exact, and $\Sigma_a$ is exact;
\item For any $kG_r$-module $W$, both $\Sigma M(W)$ and $D_aM(W)$ are free modules. In fact,
\begin{eqnarray}
\Sigma M(W) \cong M(W) \oplus M(\Res_{G_{r-1}}^{G_r}W), \hspace{1cm} DM(W) \cong M(\Res_{G_{r-1}}^{G_r}W). \label{freeshift}
\end{eqnarray}
In particular, $\Sigma$ and $D_a$ preserve $\sharp$-filtered modules.\\
\end{enumerate}
\end{proposition}

\begin{remark}
Note that if $G$ is a finite group, then $\Sigma$ and $D_a$ both preserve finitely generated $\FI_G$-modules. This is no longer the case if $G$ is infinite. It is always the case that these functors preserve being degree-wise coherent.\\
\end{remark}

Part 3 of Proposition \ref{dervprop} implies that the functors $D_a$ have left derived functors. We will follow the notation of \cite{CE} and \cite{R} and write $H_i^{D_a^b}$ for the $i$-th derived functor of $D_a^b$. One of the main insights of \cite{CE} was that the properties of the modules $H_i^{D^b}(V)$ are critical in bounding the regularity of $V$. Later, the author \cite{R} showed that the functors $H_1^{D^b}$ could be used to define a theory of depth for $\FI_G$-modules. Proofs for the following facts can be found in \cite{CE} and \cite{R}.\\

\begin{proposition}\label{dervdervprop}
Fix integers $a,b,i \geq 1$. The functors $H_i^{D_a^b}$ enjoy the following properties:
\begin{enumerate}
\item If $V$ is degree-wise coherent, then $\deg(H_i^{D_a^b}) < \infty$;
\item For any module $V$, there is an exact sequence
\[
0 \rightarrow H_1^{D_a}(V) \rightarrow V \stackrel{\tau_a}\rightarrow \Sigma_aV \rightarrow D_aV \rightarrow 0.
\]
\item If $i > b$, then $H_i^{D_a^b} = 0$.\\
\end{enumerate}
\end{proposition}

\begin{remark}
The cited sources prove these facts in the case where $a = 1$. The proofs are identical for arbitrary $a$.\\
\end{remark}

Note that the exact sequence in the second part of Proposition \ref{dervdervprop} is strongly related to torsion. This will be explored in the next section.\\

\begin{definition}
Let $V$ be a degree-wise coherent module. Then we define its \textbf{depth} to be the quantity,
\[
\depth(V) := \inf\{b \mid H_1^{D^{b+1}}(V) \neq 0\} \in \N \cup \{\infty\}.
\]
\text{}\\
\end{definition}

\begin{remark}
In \cite{LR} an alternative notion of depth is provided, which is defined in terms of the vanishing of particular $\Ext$ groups. It is shown in that paper that both notions agree with one another. Due to the emphasis on the derivative functors in this paper, we will use the above definition.\\
\end{remark}

Perhaps the most significant property of the shift functor is the following structural theorem. Note that this theorem was proven by Nagpal \cite[Theorem A]{N} in the case where $G$ is a finite group, $k$ is a Noetherian ring, and $V$ is finitely generated. It was then generalized by Nagpal and Snowden \cite{NS} to the case where $G$ is a polycyclic-by-finite group. Finally, the author \cite{R} proved the theorem to the level of generality presented here.\\

\begin{theorem}\label{nagpal}
Let $V$ be an $\FI_G$-module which is degree-wise coherent. Then for $b \gg 0$, $\Sigma_b V$ is $\sharp$-filtered.\\
\end{theorem}

\begin{definition}\label{nagpaln}
We denote the smallest $b$ for which $\Sigma_b V$ is $\sharp$-filtered by $N(V)$.\\
\end{definition}

It is natural for one to ask whether it is possible bound $N(V)$. Indeed, this was accomplished by the author in \cite[Theorem C]{R}.\\

\begin{theorem}\label{dreg}
Let $V$ be an $\FI_G$-module which is degree-wise coherent. If $V$ is not $\sharp$-filtered, then $H_1^{D^b}(V) = 0$ for $b \gg 0$, and
\[
N(V) = \max_{b}\{\deg(H_1^{D^b}(V))\}
\]
\text{}\\
\end{theorem}

One of the many consequences of Theorem \ref{nagpal} is the construction of the following complex, which we will see play a major part in the local cohomology of $\FI_G$-modules.\\

\begin{definition}\label{nonfunctcomplex}
Let $V$ be an $\FI_G$-module which is degree-wise coherent. Setting $b_{-1} := N(V)$, there is an exact sequence
\[
V \stackrel{\tau_{b_{-1}}}\rightarrow F^0 := \Sigma_bV \rightarrow D_{b_{-1}}V \rightarrow 0
\]
By Proposition \ref{dervprop}, the module $D_{b_{-1}}V$ is degree-wise coherent and is generated in strictly smaller degree than $V$. We may therefore repeat this process finitely many times to obtain the complex
\[
\C^\dt V : 0 \rightarrow V \rightarrow F^0 \rightarrow \ldots \rightarrow F^n \rightarrow 0.
\]
\text{}\\
\end{definition}

The complex $\C^\dt V$ was introduced by Nagpal in \cite[Theorem A]{N}. It was subsequently studied by Li in \cite{L2}, and by Li and the author in \cite{LR}. Note that the assignment $V \mapsto \C^\dt V$ is not functoral. Later, we will construct a uniform version of the complex $\C^\dt V$ which is functoral in $V$ (see Definition \ref{infcomplex}).\\

\section{Degree-wise coherence}

\subsection{Connections with torsion}

\begin{definition}
Let $V$ be an $\FI_G$-module. An element $v \in V_n$ is \textbf{torsion} if it is in the kernel of some - and therefore all - transition maps out of $V_n$. We say that a module $V$ is \textbf{torsion} if its every element is torsion.\\

Note that every $\FI_G$-module fits into an exact sequence of the form
\[
0 \rightarrow V_T \rightarrow V \rightarrow V_F \rightarrow 0
\]
where $V_T$ is a torsion module, and $V_F$ is torsion free.\\

The \textbf{torsion degree} of an $\FI_G$-module is the quantity
\[
\td(V) := \deg(V_T).
\]
\text{}\\
\end{definition}

The exact sequence of Proposition \ref{dervdervprop} implies that $\td(V) = \deg(H_1^{D}(V))$. Proposition \ref{dervdervprop} also tells us that $\deg(H_1^{D}(V))$ is finite. We therefore obtain the following corollary.\\

\begin{lemma}
Let $V$ be a degree-wise coherent module. Then $\td(V) < \infty$. In particular, a degree-wise coherent module $V$ is torsion if and only if $\deg(V) < \infty$.\\
\end{lemma}

We will see later that a converse of this statement is true as well. That is, if $V$ is generated in finite degree, and $\td(V) < \infty$, then $V$ is degree-wise coherent. To prove this fact, we will need the following proposition. It is, in some sense, a rephrasing of \cite[Theorem D]{CE}. Church and Ellenberg proved this for $\FI$-modules, and it was generalized to $\FI_G$-modules by the author in \cite[Theorem 3.19]{R}.

\begin{proposition} \label{dervtor}
Let $V \subseteq M$ be torsion-free $\FI_G$-modules which are generated in finite degree. Then $\td(M/V) < \infty$.\\
\end{proposition}

\begin{proof}
We have an exact sequence,
\[
0 \rightarrow V \rightarrow M \rightarrow M/V \rightarrow 0
\]
Applying the functor $D$, we obtain an exact sequence
\[
H_1^D(M) \rightarrow H_1^D(M/V) \rightarrow DV \rightarrow DM.
\]
By assumption $M$ is torsion-free, and therefore $H_1^D(M) = 0$. This implies that $H_1^D(M/V) \cong \ker(DV \rightarrow DM)$. Unpacking definitions, \cite[Theorem D]{CE} and \cite[Theorem 3.19]{R} imply that this kernel is only non-zero in finitely many degrees.\\
\end{proof}

We are now able to prove the main theorem of this section.\\

\begin{theorem}\label{altchar}
Let $V$ be an $\FI_G$-module which is generated in finite degree. Then $V$ is degree-wise coherent if and only if $\td(V) < \infty$.\\
\end{theorem}

\begin{proof}
We have already seen the forward direction. Conversely, assume that $\gd(V) < \infty$ and $\td(V) < \infty$. Then we have an exact sequence
\[
0 \rightarrow V_T \rightarrow V \rightarrow V_F \rightarrow 0
\]
where $V_T$ is torsion, and $V_F$ is torsion free. Applying the homology functor, it follows that
\[
\deg(H_1(V)) \leq \max\{\deg(H_1(V_T)), \deg(H_1(V_F))\}
\]
It is easily seen that $\deg(H_1(V_T)) < \infty$, and therefore it suffices to show that $\deg(H_1(V_F))$ is finite. In particular, we may assume without loss of generality that $V$ is torsion free.\\

Assuming that $V$ is torsion free, we have an exact sequence
\[
0 \rightarrow V \rightarrow \Sigma V \rightarrow DV \rightarrow 0
\]
where $\Sigma V$ is also torsion free. Proposition \ref{dervtor} now implies that $\td(DV) < \infty$. We also know, however, that $\gd(DV) < \gd(V) < \infty$ by Proposition \ref{dervprop}. Applying induction on the generating degree, we may assume that $DV$ is degree-wise coherent. Proposition \ref{dervdervprop} implies that $\deg(H_i^{D^b}(DV)) < \infty$ for all $i,b$.\\

Next, we claim that for all $i,b \geq 1$, $H_i^{D^b}(DV) \cong H_i^{D^{b+1}}(V)$. To see this, we compute the derived functors of $D^{b+1}$, when viewed as the composition $D^{b} \circ D$. Proposition \ref{dervdervprop} implies the Grothendieck spectral sequence associated to this composition only has two rows. It therefore degenerates to the long exact sequence 
\[
\ldots \rightarrow H_{i-1}^{D^b}(H_1^D(V)) \rightarrow H_i^{D^b}(DV) \rightarrow H_{i}^{D^{b+1}}(V) \stackrel{\partial}\rightarrow H_{i-1}^{D^b}(H_1^D(V)) \rightarrow \ldots
\]
The fact that $V$ is torsion-free implies $H_1^D(V) = 0$, and therefore $H_i^{D^b}(DV) \cong H_i^{D^{b+1}}(V)$ for all $i$.\\

Recall that we have shown that $DV$ is degree-wise coherent. The above isomorphisms therefore imply that $\deg(H_i^{D^b}(V)) < \infty$ for all $i,b$. Theorem \ref{nagpaln} now implies that $\Sigma_b V$ is $\sharp$-filtered for $b \gg 0$. In particular, we have an exact sequence
\[
0 \rightarrow V \rightarrow \Sigma_{N(V)}V \rightarrow D_{N(V)}V \rightarrow 0
\]
By assumption $V$ is generated in finite degree, and therefore $D_{N(V)}V$ is degree-wise coherent. Applying the homology functor, and using Theorems \ref{homacyclic} and \ref{finreg}, we conclude that $\deg(H_1(V)) < \infty$, as desired.\\
\end{proof}

\begin{remark}
The author's interest in proving the above theorem was heavily influenced by recent work of Li \cite{L3}. In that work, Li argues the forward direction of the theorem, and leaves the converse as a conjecture. The author would like to thank Professor Li for pointing him in the direction of this problem.\\
\end{remark}

\begin{remark}
It is important that one develop an intuition for why one would suspect Theorem \ref{altchar} is true. In the work of Li and the author \cite[Theorem F]{LR}, it is shown that the regularity of a finitely generated $\FI_G$-module, where $G$ is finite and $k$ is Noetherian, can be bound in terms of the torsion degrees of its local cohomology modules (see Definition \ref{lc}). Li has shown that the higher local cohomology modules can be bounded entirely in terms of the generating degree \cite{L2}. Put together, it follows that the regularity of a finitely generated $\FI_G$-module is bounded by a constant depending only on its torsion degree and its generating degree. Theorem \ref{altchar} implies that these bounds on regularity will continue to hold even if we do not assume that the module is finitely generated.\\
\end{remark}

\subsection{The category $\FI_G\Mod^{coh}$}

In this section we consider the category of degree-wise coherent modules, and examine some of its technical properties. The main result of this section will be to show that $\FI_G\Mod^{coh}$ is abelian. We once again note that the category of finitely generated $\FI_G$-modules is only known to be abelian when $k$ is Noetherian, and $G$ is polycyclic-by-finite. This would seem to indicate that the property of being degree-wise coherent is often times better suited for homologically flavored questions about $\FI_G$-modules.\\

One recurring theme throughout the proofs in this section is Theorem \ref{altchar}. This theorem tells us that the property of being degree-wise coherent can be partially checked on the maximal torsion submodule. This will allow us to prove non-obvious facts about submodules of degree-wise coherent submodules. One example of this is the following.\\

\begin{proposition}\label{coh}
Let $V$ be a degree-wise coherent $\FI_G$-module, and let $V' \subseteq V$ be a submodule which is generated in finite degree. Then $V'$ is also degree-wise coherent.\\
\end{proposition}

\begin{proof}
Because $V'$ is a submodule of $V$, we must have $\td(V') \leq \td(V)$. Theorem \ref{altchar} now implies the proposition.\\
\end{proof}

Note that the above proposition justifies the terminology of coherence. Recall that a module $M$ over a commutative ring $R$ is said to be coherent if it is finitely presented, and every finitely generated submodule of $M$ is also finitely presented. It is well known that a module over a coherent ring is finitely presented if and only if it is coherent. When $k$ is a field of characteristic 0, Sam and Snowden's language of twisted commutative algebras imply that the category of $\FI$-modules is equivalent to the category of $GL_\infty$-equivariant modules over a polynomial ring in infinitely many variables \cite{SS3}. A polynomial ring in infinitely many variables over a field is coherent, and therefore Proposition \ref{coh} can be heuristically thought of as a consequence of this.\\

\begin{proposition}\label{needtwo}
Let,
\[
0 \rightarrow V' \rightarrow V \rightarrow V'' \rightarrow 0
\]
be an exact sequence of $\FI_G$-modules. Then any two of $V',V,$ or $V''$ are degree-wise coherent only if the third is as well.\\
\end{proposition}

\begin{proof}
The above exact sequence induces the exact sequence,
\[
H_2(V'') \rightarrow H_1(V') \rightarrow H_1(V) \rightarrow H_1(V'') \rightarrow H_0(V') \rightarrow H_0(V) \rightarrow H_0(V'') \rightarrow 0.
\]
This implies the collection of bounds,

\begin{eqnarray}
\hd_1(V) \leq \max\{\hd_1(V'),\hd_1(V'')\}, \hspace{1cm} \gd(V) \leq \max\{\gd(V'),\gd(V'')\}\\
\hd_1(V') \leq \max\{\hd_2(V''),\hd_1(V)\}, \hspace{1cm} \gd(V') \leq \max\{\hd_1(V''),\gd(V)\}\\
\hd_1(V'') \leq \max\{\gd(V'),\hd_1(V)\}, \hspace{1cm} \gd(V'') \leq \gd(V).
\end{eqnarray}

If $V'$ and $V''$ are degree-wise coherent, then the first pair of bounds immediately implies the same about $V$. If we instead assume that $V''$ and $V'$ are degree-wise coherent, then Theorem \ref{finreg} implies that $\hd_2(V'') < \infty$. The second pair of bounds now imply that $V'$ is degree-wise coherent. Finally, if $V'$ and $V$ are degree-wise coherent then the third pair of bounds imply that $V''$ must be as well.\\
\end{proof}

This is all we need to prove the main theorem of this section.\\

\begin{theorem}\label{abelian}
The category $\FI_G\Mod^{coh}$ is abelian.\\
\end{theorem}

\begin{proof}
The only thing that needs to be checked is that $\FI_G\Mod^{coh}$ permits images, kernels and cokernels. That is, if $\phi:V \rightarrow V'$ is a morphism of degree-wise coherent modules, then we must show that $\ker(\phi),\im(\phi)$ and $\coker(\phi)$ are all degree-wise coherent. We have a pair of exact sequences
\begin{eqnarray*}
0 \rightarrow \ker(\phi) \rightarrow V \rightarrow \im(\phi) \rightarrow 0\\
0 \rightarrow \im(\phi) \rightarrow V' \rightarrow \coker(\phi) \rightarrow 0\\
\end{eqnarray*}
The module $\im(\phi)$ is generated in finite degree because it is a quotient of $V$, and $\td(\im(\phi)) < \infty$ because it is a submodule of $V'$. Theorem \ref{altchar} implies that $\im(\phi)$ is degree-wise coherent, whence $\ker(\phi)$ and $\coker(\phi)$ are as well by Proposition \ref{needtwo}.\\
\end{proof}

\begin{remark}
Li has also independently proven this theorem in his work \cite[Proposition 3.4]{L3}. His methods do not use Theorem \ref{altchar}.\\
\end{remark}

\section{Applications}

In this half of the paper, we consider applications of the machinery developed in previous sections. To start, we will define the infinite shift and derivative functors. Using these functors, we will describe a local cohomology theory for degree-wise coherent $\FI_G$-modules. Finally, we finish by proving a kind of local duality theorem for $\FI_G$-modules.\\

\subsection{The infinite shift and derivative functors}\label{infderv}

\begin{definition}\label{infderv2}
Let $V$ be an $\FI_G$-module. For each positive integer $a$, the transition map $(f^{n+a},\mathbf{1})_\as$, induced by the pair of the standard inclusion $f^{n+a}:[n+a] \rightarrow [n+a+1]$ and the trivial map into $G$, gives a map $\Sigma_a V \rightarrow \Sigma_{a+1}V$. The \textbf{infinite shift} of $V$ is the direct limit
\[
\Sigma_\infty V := \lim_{\rightarrow} \Sigma_a V
\]
The maps $(f^{n+a},\mathbf{1})_\as$ also induce maps $D_aV \rightarrow D_{a+1}V$. The \textbf{infinite derivative} of the module $V$ is the direct limit
\[
D_\infty V := \lim_{\rightarrow} D_aV
\]
\end{definition}

One should immediately note that if $V$ is finitely generated, then neither $\Sigma_\infty V$, nor $D_\infty V$ are necessarily finitely generated. These functors do preserve degree-wise coherence, as we shall now prove.\\

\begin{proposition}\label{infprop}
The infinite shift and derivative functors enjoy the following properties:
\begin{enumerate}
\item $\Sigma_\infty$ is exact, and $D_\infty$ is right exact;
\item for all $\FI_G$-modules $V$, there is an exact sequence
\[
V \rightarrow \Sigma_\infty V \rightarrow D_\infty V \rightarrow 0.
\]
$V$ is torsion-free if and only if the map $V \rightarrow \Sigma_\infty V$ is injective;
\item for any $kG_n$-module $W$, $\Sigma_\infty M(W) \cong M(W) \oplus Q$ where $Q$ is some free-module generated in degree $< r$, while $D_\infty M(W) \cong Q$. In particular, both the infinite shift and derivative functors preserve $\sharp$-filtered objects;
\item if $\gd(V) \leq d$ is finite, then $\gd(\Sigma_\infty V) \leq d$ and $\gd(D_\infty V) < d$;
\item if $V$ is degree-wise coherent, then $\Sigma_\infty V$ is $\sharp$-filtered, and $D_\infty V$ is degree-wise coherent.\\
\end{enumerate}
\end{proposition}

\begin{proof}
The first statement follows from Proposition \ref{dervprop}, as well as the exactness of filtered colimits.\\

Write $\omega$ for the poset category of the natural numbers. We define the functors $F_i: \omega \rightarrow \FI_G\Mod$, $i = 1,2,3$ as follows:
\[
F_1(a) = V, \hspace{1cm} F_2(a) = \Sigma_aV, \hspace{1cm} F_3(a) = D_aV.
\]
Note that $F_1$ maps all morphisms of $\omega$ to the identity, while $F_2$ and $F_3$ map the morphisms of $\omega$ to the previously discussed maps $\Sigma_aV \rightarrow \Sigma_{a+1} V$ and $D_aV \rightarrow D_{a+1}V$. Then all relevant definitions imply there is an exact sequence
\[
F_1 \rightarrow F_2 \rightarrow F_3 \rightarrow 0
\]
Applying the exact direct limit functor to this exact sequence implies the first half of the second claim. If $V$ is torsion free, then the map $F_1 \rightarrow F_2$ is exact by definition of torsion, and this will be preserved after taking direct limits. Conversely, assume that $V$ has torsion. In particular, there is an element $v \in V_n$ for some $n$, such that $v$ is in the kernel of some transition map $V_n \rightarrow V_m$. In this case, every transition map to $V_r$, with $r \geq m$, will also contain $v$ in its kernel. In particular, $v$ will be an element in the kernel of the maps $V \rightarrow \Sigma_a V$ for all $a > 0$. This implies that the element $v$ is in the kernel of the map $V \rightarrow \Sigma_\infty V$.\\

The fact that $\Sigma_\infty M(W)$ takes the prescribed form follows immediately from Proposition \ref{dervprop} and (\ref{freeshift}). The statement about the infinite derivative follows from the second part of this proposition.\\

The fourth statement follows from the first statement and the third.\\

The fifth statement follows from the fourth, as well as Theorem \ref{nagpal}.\\
\end{proof}

While the infinite shift and derivative functors may be harder to compute than their finite counter-parts, they allow us to more uniformly state certain theorems. For instance, we will see that infinite shifts can be used to fix the issue of functoriality of the complex $\C^\dt V$. We will also see that the infinite derivative functor can be used to prove a kind of local duality for $\FI_G$-modules.\\

The above proposition implies that the functors $\Sigma_\infty$ and $D_\infty$ can be considered as endofunctors of the abelian category $\FI_G\Mod^{coh}$. This proposition also tells us that $D_\infty$ admits left derived functors in this category.\\

\begin{definition}
For each $b \geq 1$, we will write $H_i^{D_\infty^b}:\FI_G\Mod^{coh} \rightarrow \FI_G\Mod^{coh}$ to denote the $i$-th derived functor of $D_\infty^b$.\\
\end{definition}

\begin{proposition}\label{infdervdervprop}
The functors $H_i^{D_\infty^b}$ enjoy the following properties:
\begin{enumerate}
\item for all degree-wise coherent modules $V$, there is an exact sequence
\[
0 \rightarrow H_1^{D_\infty}(V) \rightarrow V \rightarrow \Sigma_\infty V \rightarrow D_\infty V \rightarrow 0;
\]
In particular, if $V$ is torsion free, then $H_1^{D_\infty}(V) = 0$;
\item If $V$ is $\sharp$-filtered, then $H_i^{D_\infty^b}(V) = 0$ for all $i,b \geq 1$;
\item for all degree-wise coherent modules $V$, and all $b,i\geq 1$, $\deg(H_i^{D_\infty^b}(V)) < \infty$.\\
\end{enumerate}
\end{proposition}

\begin{proof}
Let
\[
0 \rightarrow K \rightarrow F \rightarrow V \rightarrow 0
\]
be a presentation for $V$. Then we have a commutative diagram with exact rows
\[
\begin{CD}
@. D_\infty^b(K) @>>> \Sigma_\infty D_\infty^b(K) @>>> D_{\infty}^{b+1}(K) @>>> 0\\
@. @VVV                      @VVV                                 @VVV\\
0 @>>> D_\infty^b(F) @>>> \Sigma_\infty D_\infty^b(F) @>>> D_\infty^{b+1}(F) @>>> 0\\
\end{CD}
\]
Note that the second row is exact on the left, as $D_\infty^b(F)$ is $\sharp$-filtered, and therefore it is torsion free.
 Applying the snake lemma, we obtain a long exact sequence
\begin{eqnarray}
H_1^{D_\infty^b}(V) \rightarrow \Sigma_\infty H_1^{D_\infty^b}(V) \rightarrow H_1^{D_\infty^{b+1}}(V) \rightarrow D_\infty^b(V) \rightarrow \Sigma_{\infty}D_\infty^b(V) \rightarrow D_{\infty}^{b+1}(V) \rightarrow 0 \label{exactder}
\end{eqnarray}

Now assume that $b = 0$. In this case the above becomes the claimed exact sequence of the first part of the proposition.\\

We can prove the second statement by induction on $b$. Note that Theorem \ref{homacyclic} implies that any presentation of a $\sharp$-filtered module will necessarily have a $\sharp$-filtered first syzygy. It follows that it suffices to prove the second claim in the proposition for $i = 1$. Because $\sharp$-filtered objects are torsion free, the first part of this proposition implies the claim for $b = 1$. Otherwise, the exact sequence (\ref{exactder}) degenerates to,
\[
0 \rightarrow H_1^{D_\infty^{b+1}}(V) \rightarrow D_\infty^b(V) \rightarrow \Sigma_{\infty}D_\infty^b(V)
\]
Using the fact that the infinite derivative of a $\sharp$-filtered object is still $\sharp$-filtered, as well as the fact that $\sharp$-filtered objects are torsion free, we obtain our desired vanishing.\\

Straight forward homological dimension shifting arguments imply that it suffices to prove the third claim for $i = 1$. We proceed by induction on $b$. If $b = 1$, then the first statement along with Theorem \ref{altchar} imply that $H_1^{D_\infty}(V)$ has finite degree. Assume that the statement is true for some integer $b \geq 1$, and consider the sequence (\ref{exactder}). By induction we know that $H_1^{D_\infty^b}(V)$ has finite degree, and therefore $\Sigma_\infty H_1^{D_\infty^b}(V) = 0$. The above sequence will simplify to
\[
0 \rightarrow H_1^{D_\infty^{b+1}}(V) \rightarrow D_\infty^bV \rightarrow \Sigma_{\infty}D_\infty^bV \rightarrow D_{\infty}^{b+1}V \rightarrow 0.
\]
Proposition \ref{infprop} implies that $D_\infty^b(V)$ is degree-wise coherent, and therefore it has finite torsion degree by Theorem \ref{altchar}. We conclude that $H_1^{D_\infty^{b+1}}(V)$ has finite degree, as desired.\\
\end{proof}

To finish this section, we define an improved version of the complex $\C^\dt V$. This new complex will share almost all of $\C^\dt V$'s most important properties, while having the advantage of being functoral in $V$.\\

\begin{definition}\label{infcomplex}
Let $V$ be a degree-wise coherent $\FI_G$-module. Then Theorem \ref{nagpal} and Proposition \ref{infprop} imply that $\Sigma_\infty V$ is $\sharp$-filtered, and that there is an exact sequence
\[
V \rightarrow \Sigma_\infty V = F^0 \rightarrow D_\infty V \rightarrow 0
\]
where $D_\infty V$ is also degree-wise coherent with strictly smaller generating degree. Repeating this process, we obtain a complex
\[
\C_\infty^\dt V: 0 \rightarrow V \rightarrow F^0 \rightarrow F^1 \rightarrow \ldots \rightarrow F^n \rightarrow 0.
\]
Note that by construction,
\[
H^i(\C_\infty^\dt V) \cong \ker(D_{\infty}^{i+1} V \rightarrow \Sigma_\infty D_{\infty}^{i+1} V) \cong H_1^{D_\infty}(D_{\infty}^{i+1} V)
\]
where $D_\infty^0$ is the identity functor by convention. In particular, the cohomology modules of $\C_\infty^{\dt}$ all have finite degree.\\
\end{definition}

\subsection{Local Cohomology} \label{locsection}

In this section, we record results about local the local cohomology of the modules in $\FI\Mod^{coh}$. These facts were proven about finitely generated modules in \cite{LR}, and the proofs from that paper will work in this context as well, thanks to Theorems \ref{abelian} and \ref{nagpal}. The fact that these two results imply that the work of \cite{LR} will hold for degree-wise coherent modules was also noted by Li in \cite{L3}.\\

\begin{definition}\label{lc}
Recall that every $\FI_G$-module $V$ fits into an exact sequence
\[
0 \rightarrow V_T \rightarrow V \rightarrow V_F \rightarrow 0
\]
where $V_T$ is torsion, and $V_F$ is torsion free. The \textbf{0-th local cohomology functor} is defined by
\[
H_\mi^0(V) := V_T
\]
The category $\FI_G\Mod$ is Grothendieck, and therefore we can define the right derived functors of $H_\mi^0$. The $i$-th derived functor of $H^0_\mi$ is denoted by $H^i_\mi$, and is known as the \textbf{$i$-th local cohomology functor}.\\
\end{definition}

One of the main results of the paper \cite[Theorem E]{LR}, is that, when working over a Noetherian ring, $H^i_\mi(V)$ is finitely generated whenever $V$ is. In this work we will show that $H^i_\mi(V)$ is degree-wise coherent whenever $V$ is. To do so, we first record the following alternative definition of local cohomology.\\

\begin{definition}
For each integer $r \geq 0$, and each integer $n \geq 1$, we define the module $M(r)/\mi^n M(r)$ to be the quotient of $M(r)$ by the submodule generated by $M(r)_{r+n}$. Then we define the functor $\fiHom(k\FI_G/\mi^n,\dt):\FI_G\Mod \rightarrow \FI_G\Mod$ by
\[
\fiHom(k\FI_G/\mi^n,V)_r := \Hom_{\FI_G\Mod}(M(r)/\mi^n M(r),V)
\]
Note that a map $M(r)/\mi^nM(r) \rightarrow V$ is determined by a choice of element $V_r$, which is in the kernel of all transition maps into $V_{r+n}$. Given such a map $\phi:M(r)/\mi^nM(r) \rightarrow V$, and a morphism $(f,g):[r] \rightarrow [m]$ in $\FI_G$, we define $(f,g)_\as\phi$ to be the map $M(m)/\mi^nM(m) \rightarrow V$ which sends the identity in degree $m$ to $(f,g)_\as(\phi(id_r))$. This defines an $\FI_G$-module stucture on $\fiHom(k\FI_G/\mi^n,V)$. We use $\fiExt^i(k\FI_G/\mi^n,\dt)$ to denote the $i$-th derived functor of $\fiHom(k\FI_G/\mi^n,\dt)$.\\
\end{definition}

One important observation is that for each $r \geq 0$ and $n \geq 1$, there is a map
\[
M(r)/\mi^{n+1}M(r) \rightarrow M(r) / \mi^n M(r).
\]
This induces maps $\Hom_{\FI_G\Mod}(M(r)/\mi^n M(r),V) \rightarrow \Hom_{\FI_G\Mod}(M(r)/\mi^{n+1} M(r),V)$, which one may check are compatible with the induced maps of $\fiHom(k\FI_G/\mi^n,V)$. In particular, for any $V$ we obtain a morphism of $\FI_G$-modules
\[
\fiHom(k\FI_G/\mi^n,V) \rightarrow \fiHom(k\FI_G/\mi^{n+1},V).
\]
This also gives us maps
\[
\fiExt^i(k\FI_G/\mi^n,V) \rightarrow \fiExt^i(k\FI_G/\mi^{n+1},V)
\]
for each $i \geq 0$. This justifies the following proposition.\\

\begin{proposition}
There is an isomorphism of functors,
\[
H^0_\mi(\dt) \cong \lim_{\rightarrow} \fiHom(k\FI_G/\mi^n,V),
\]
inducing isomorphisms of derived functors
\[
H^i_\mi(\dt) \cong \lim_{\rightarrow} \fiExt(k\FI_G/\mi^n,V)
\]
\text{}\\
\end{proposition}

Using this alternative description, one then goes on to prove the following acyclicity results.\\

\begin{proposition}\label{acyclic}
Let $V$ be degree-wise coherent. If $V$ is either a torsion module, or a $\sharp$-filtered module, then
\[
H^i_\mi(V) = 0
\]
for all $i \geq 1$.\\
\end{proposition}

Next, we recall the complex $\C^\dt_\infty V$. By construction this complex is comprised of $\sharp$-filtered modules in its positive degrees, and its cohomologies are all degree-wise coherent torsion modules. The above proposition can therefore be used to prove the following.\\

\begin{theorem}\label{cohomologycoh}
Let $V$ be a degree-wise coherent module. Then there are isomorphisms for all $i \geq 0$,
\[
H^i_\mi(V) \cong H^{i-1}(\C^\dt_\infty V)
\]
In particular, if $V$ is degree-wise coherent, then the same is true of its local cohomology modules.\\
\end{theorem}

This theorem has a long list of consequences, some of which we list now.\\

\begin{corollary}
Let $V$ be a degree-wise coherent module. Then $V$ is acyclic with respect to local cohomology if and only if there is an exact sequence
\[
0 \rightarrow V_T \rightarrow V \rightarrow V_F \rightarrow 0
\]
where $V_T$ is a torsion module, and $V_F$ is $\sharp$-filtered.\\
\end{corollary}

\begin{corollary} \label{dimension}
Let $V$ be a degree-wise coherent module. Then $H^i_\mi(V) = 0$ for $i \gg 0$, while 
\[
\depth(V) = \inf\{i \mid H^i_\mi(V) \neq 0\}
\]
\text{}\\
\end{corollary}

\begin{definition}\label{dim}
Let $V$ be a degree-wise coherent module which is not $\sharp$-filtered. Then Corollary \ref{dimension} implies that there is a largest $i$ for which $H^i_\mi(V) \neq 0$. We define the \textbf{dimension} of the module $V$ to be the quantity,
\[
\dim_{\FI_G}(V) := \sup\{i \mid H^i_\mi(V)\}.
\]
If $V$ is $\sharp$-filtered, then we set $\dim_{\FI_G}(V) = \infty$.\\
\end{definition}

\begin{corollary}
Let $V$ be a degree-wise coherent module. Then,
\[
N(V) = \max_i\{\deg(H^i_\mi(V))\} + 1,
\]
whenever $V$ is not $\sharp$-filtered.\\
\end{corollary}

\begin{corollary}\label{regconj}
Let $V$ be a degree-wise coherent module. Then,
\[
\reg(V) \leq \max_i\{\deg(H^i_\mi(V)) + i\}.
\]
\text{}\\
\end{corollary}

The reader might have noticed that Corollary \ref{regconj} looks very similar to a classic result from the local cohomology theory of the polynomial ring. Indeed, it is the belief of the author that the following is true.\\

\begin{conjecture}
Let $V$ be a degree-wise coherent module. Then,
\begin{eqnarray}
\reg(V) = \max_i\{\deg(H^i_\mi(V)) + i\} \label{regconjj}
\end{eqnarray}
\end{conjecture}

\begin{remark}
Note that the above conjecture would be false if our definition of $\reg(V)$ included the $0$-th homological degree. Indeed, Proposition \ref{acyclic} implies that $\sharp$-filtered modules are torsion free acyclics with respect to local cohomology, and therefore the right hand side of (\ref{regconjj}) is always $-\infty$ for $\sharp$-filtered modules.\\
\end{remark}

Note that as of the writing of this paper, not much is known about this conjecture. It was shown to be true for torsion modules by Liang Gan, and Li in their paper \cite{GL}.\\

To finish this section, we more closely examine the relationship between the infinite derivative and local cohomology.\\

\begin{proposition} \label{induction}
Let $H_i^{D_\infty^b}:\FI_G\Mod^{coh} \rightarrow \FI_G\Mod^{coh}$ denote the $i$-th derived functor of $D_\infty^b$. Then for all $i,b\geq 1$ there are natural isomorphisms of functors
\[
H^{D_\infty^{b}}_{i} \cong H^{D_\infty^{b-1}}_{i-1} \cong \ldots \cong H^{D_\infty^{b-i+1}}_{1} \cong H^{D_\infty}_{1} \circ D_\infty^{b-i}.
\]
\text{}\\
\end{proposition}

\begin{proof}
Consider the Grothendieck spectral sequence associated to the composition $D_\infty \circ D^{b-1}_\infty$. Note that Proposition \ref{infdervdervprop} implies that $H^{D_\infty}_i(V) = 0$ for all $i > 1$, and all degree-wise coherent modules $V$, and therefore this spectral sequence only has two columns. The spectral sequence will therefore degenerate to the collection of short exact sequences
\[
0 \rightarrow D_\infty H^{D_\infty^{b-1}}_i(V) \rightarrow H^{D_\infty^{b}}_i(V) \rightarrow H^{D_\infty}_1(H^{D_\infty^{b-1}}_{i-1}(V)) \rightarrow 0.
\]
Proposition \ref{infdervdervprop} tells us that $H^{D_\infty^{b-1}}_i(V)$ has finite degree, and therefore the left most term in these exact sequences is always zero. This same proposition also implies that $H^{D_\infty}_1(H^{D_\infty^{b-1}}_{i-1}(V)) \cong H^{D_\infty^{b-1}}_{i-1}(V)$ whenever $i > 1$. Naturality of the isomorphisms $H^{D_\infty^{b}}_i(V) \cong H^{D_\infty^{b-1}}_{i-1}(V)$ follows from the naturality of the Grothendieck spectral sequence. The result now follows by induction.\\
\end{proof}

Theorem \ref{cohomologycoh} tells us that the local cohomology modules of a degree-wise coherent module $V$ can be computed as the torsion submodules of the infinite derivatives of $V$. Proposition \ref{induction} directly relates these torsion modules to the derived functors of these infinite derivatives. Putting everything together, we have proven the following theorem. One may think of this as a kind of ``local duality,'' as it relates the local cohomology functors to the derived functors of some right exact functor.\\

\begin{theorem}
Let $V$ be a degree-wise coherent module of dimension $d$. Then there are isomorphisms for all $i \geq 1$,
\[
H_i^{D_\infty^{d+1}}(V) \cong H^{d+1-i}_\mi(V).
\]
\text{}\\
\end{theorem}

\begin{proof}
Theorem \ref{cohomologycoh} and Proposition \ref{induction} imply
\[
H_i^{D_\infty^{d+1}}(V) \cong H_1^{D_\infty}(D_\infty^{d+1-i}V) \cong H^{d-i}(\C_\infty^\dt V) \cong H^{d+1-i}_\mi(V).
\]
\text{}\\
\end{proof}

\end{document}